\documentclass[11pt]{amsart}
\usepackage{amsmath,amsthm,amsfonts,amssymb,amscd}
\usepackage[latin1]{inputenc}
\usepackage{palatino}

{\newtheorem{Theorem}{Theorem}
\newtheorem{Corollary}{Corollary}
\newtheorem{Proposition}{Proposition}
\newtheorem{Lemma}{Lemma}
\newtheorem{Claim}{Claim}
\theoremstyle{Definition}
\newtheorem{Definition}{Definition}
\newtheorem{Example}{Example}
\newtheorem{Question}{Question}

\theoremstyle{Remark}
\newtheorem{Remark}{Remark}

\headsep=1cm 

\oddsidemargin=30pt \evensidemargin=0pt \textwidth=15truecm
\textheight=21truecm

\def\leaderfill{\leaders\hbox to .8em{\hss .\hss}\hfill}
\def\_#1{{\lower 0.7ex\hbox{}}_{#1}}

\def\esimo{${}^{\text{\b o}}$}

\newcommand\virt{\rm{virt}}

\def\fa{{\mathcal{F}}}

\def\po{{\partial}}

\def\te{{\theta}}

\def\Om{{\Omega}}
\def\vr{{\varphi}}
\def\ga{{\gamma}}
\def\Ga{{\Gamma}}
\def\la{{\lambda}}

\def\ov{\overline}
\def\al{{\alpha}}

\def\lg{{\langle}}
\def\rg{{\rangle}}

\def\re{{\mathbb{R}}}
\def\bz{{\mathbb{Z}}}
\def\bq{{\mathbb{Q}}}
\def\bd{{\mathbb{D}}}
\def\bc{{\mathbb{C}}}
\def\bn{{\mathbb{N}}}

\def\dim{\operatorname{{dim}}}

\def\Hol{\operatorname{{Hol}}}

\def\res{\operatorname{{res}}}
\def\hot{\operatorname{h.o.t.}}
\def\Id{\operatorname{{Id}}}

\def\Diff{\operatorname{{Diff}}}
\def\sing{\operatorname{{sing}}}
\def\sing{\operatorname{{sing}}}

\renewcommand{\thefootnote}

\title[On  holomorphic flows on  Stein surfaces]{On  holomorphic flows on  Stein surfaces: transversality, dicriticalness and stability}

\author{T. Ito \qquad B. Sc\'ardua \qquad Y. Yamagishi}

\address{T. Ito and Y. Yamagishi: Department of Natural Science, Ryukoku University,
Fushimi-ku, Kyoto 612, JAPAN}

\address{B. Scardua: Inst. Matem\'atica,  Universidade Federal do Rio de Janeiro.   Caixa Postal 68530,
 Rio de Janeiro-RJ,  21.945-970 BRAZIL}

\date{}
\begin{document}

\maketitle

\begin{abstract}

We study the classification of the pairs $(N, \,X)$ where $N$ is a
Stein surface and $X$ is a complete holomorphic vector field with
isolated singularities on $N$. We describe the role of transverse
sections in the classification of $X$ and give necessary and
sufficient conditions on $X$ in order to have $N$ biholomorphic to
$\bc^2$. As a sample of our results, we prove that $N$ is
biholomorphic to $\bc^2$ if $H^2(N,\mathbb Z)=0$, $X$ has a finite
number of singularities and exhibits a non-nilpotent  singularity
with three separatrices or, equivalently, a singularity with first
jet of the form $\lambda_1 \, x\frac{\po}{\po x} + \lambda_2 \,
y\frac{\po}{\po y}$ where $\lambda_1 / \lambda_2 \in \mathbb Q_+$. We also study flows with many periodic orbits,  in a sense we will make clear, proving they admit a meromorphic first integral or they exhibit some special periodic orbit, whose holonomy map is a non-resonant non-linearizable diffeomorphism map.  Finally, we apply our results together with more classical techniques of holomorphic foliations on algebraic surfaces to study the case of flows on affine algebraic surfaces. We suppose the flow is generated by an algebraic vector field. Such flows are then proved, under some undemanding conditions on the singularities, to be given by closed rational linear one-forms or admit rational first integrals. 
\end{abstract}

\tableofcontents

\thispagestyle{empty}

\section{Introduction}
\label{section:introduction}

In a pioneering work, M. Suzuki  introduces the study of
holomorphic flows on  Stein surfaces by the use of techniques of
Holomorphic Foliations, Potential Theory and Theory of Analytic
Spaces (cf. \cite{Suzuki1} and \cite{Suzuki2}). In his original papers the author proves 
that {\sl there is a notion of typical orbit} and {\it if this
orbit is biholomorphic to $\bc^*$ then the corresponding
holomorphic foliation has a meromorphic first 
integral}. 
Another achievement in \cite{Suzuki1} is that {\sl orbits
biholomorphic to $\bc^*$ have analytic closure of dimension one},
i.e., {\sl are contained in invariant analytic curves}. More
recently these facts have been used as main ingredients in the
classification of complete  {\sl polynomial} vector fields in the
affine space $\bc^2$ (\cite{[Brunella]},
\cite{[Cerveau-Scardua]}).  This paper is devoted to the study of
geometric properties in connection with the classification  of
holomorphic flows on Stein manifolds of dimension two. More
precisely, we resume the study of holomorphic flows on connected
Stein surfaces under the hypothesis of existence of a suitable
real transverse section (see \cite{[Ito-Scardua 4]}) and extend
this study to the classification of  the pairs $N, X$ where $N$ is
a Stein surface and $X$ is a complete holomorphic vector field on
$N$ having a suitable isolated singularity. Given a holomorphic
vector field $X$ with isolated singularities on a complex surface
$N$ we denote by $\fa(X)$ the one-dimensional holomorphic
foliation on $N$ whose leaves are the nonsingular orbits of $X$
and with singular set $\sing(\fa(X))=\sing(X)$. A singularity
$p\in \sing(X)$ is {\it dicritical} if for some neighborhood $p\in
V$ there are infinitely many orbits of $X\big|_V$ accumulating
only at $p$. The closure of such a local leaf is an invariant
analytic curve called a {\it separatrix} of $X$ through $p$. An
isolated singularity $P\in \sing(X)$ is in the {\it Poincar\'e
domain} if $DX(P)$ is nonsingular with eigenvalues $\lambda_1,
\lambda_2$ satisfying $\lambda_1 / \lambda_2 \in \mathbb C
\setminus \mathbb R_-$ (\cite{Arnold}).  Let $\ov D^2(1)\subset
\bc^2$ be the closed unit disc given in coordinates $(z_1,z_2)$ by
$\sum\limits_{j=1}^2 |z_j|^2 \leq 1$. Our first  result reads as:

\begin{Theorem} \label{Theorem:main}  Let $X$ be a {\rm complete} holomorphic
vector field with isolated singularities on a connected Stein
surface $N$ with $H^2(N,\mathbb Z)=0$. Suppose that  $\fa(X)$  is
transverse to the boundary of a compact domain $D \subset N$
biholomorphic to the closed unit $2$-disc
 $\ov D^2 (1) \subset \bc^2$. Then there is a holomorphic
diffeomorphism $\Phi$ of an invariant neighborhood $U$ of $D$ in $N$ onto
$\Phi(U) = \bc^2$ such that $\Phi$ conjugates $X$ to a linear
vector field $X_\la = x\,\frac{\po}{\po x} + \la y\,\frac{\po}{\po
y}$\,, $\la \in \bc\backslash\re_-$\,, or to a Poincar\'e-Dulac
normal form vector field $Y_n = y\,\frac{\po}{\po y} +
(nx+y^n)\,\frac{\po}{\po x}$\,, $n \in \bn\backslash\{1\}$ in $U$.
We can take $U=N$ if and only if every leaf of $\fa(X)$ crosses
the boundary $\po D$.
\end{Theorem}

Theorem~\ref{Theorem:main} generalizes a result in
\cite{[Ito-Scardua 4]} and is a main ingredient in our approach.
As we will see in Theorem~\ref{Theorem:equivalences} below, we can
take  $U=N$ if $X\big|_U$ is conjugate to $ X_\la$ with $\la \in
\bq_+$ or, equivalently, if $X$ has some non-nilpotent or simple
(i.e., nondegenerate) dicritical singularity in $N$ (the dicritical last part is in \cite{Scarduacomplexvariables}). Before
stating our second main result we shall introduce some terminology
motivated by \cite{[Ito-Scardua 4]} and classical notions from
Geometric Theory of Foliations (\cite{Godbillon}).

\begin{Definition}
{\rm Given a holomorphic vector field $X$ on a surface $N$ we
shall say that:

\begin{itemize}
\item[{(1)}] $X$ has the {\it stability property} if it is complete and
exhibits a periodic   closed orbit in $N$ having finite holonomy group.

\item[{(2)}] $X$ has the {\it global transversality property} if
one of the conditions below is satisfied:
\begin{itemize}\itemsep=-1pt \item[\rm(2.i)] $X$ is transverse to
$\po D_j$\,, $\forall\,j\in \mathbb N$ where
$\bigcup\limits_{j\in\mathbb N} D_j = N$ is an exhaustion of $N$
by domains biholomorphic to the closed $2$-disc $\ov D^2(1)
\subset \bc^2$. \item[\rm(2.ii)] $X$ is transverse to the boundary
$\po D$ of a compact domain $D \subset N$ biholomorphic to the
closed unit $2$-disc $\ov{D}^2(1) \subset \bc^2$ and such that
every leaf of $\fa(X)$ crosses $\po D$.
\end{itemize}

\item[{(3)}] $X$ has the {\it first jet resonance property} if one the
conditions below is satisfied:
\begin{itemize}
\item[\rm(3.i)] $X$ has a linearizable singularity of type $X_\la
= x\,\frac{\po}{\po x} + \la y\,\frac{\po}{\po y}$\,, $\la =
k/\ell \in \bq_+$\,. \item[\rm(3.ii)] $X$ has a non-nilpotent
singularity with first jet of the form $X_\la$\,, $\la \in
\bq_+$\,.
\end{itemize}

\item[{(4)}] $X$ has the {\it dicriticalness property} if fulfills
one of the following conditions:
\begin{itemize}
\item[\rm(4.1)] $X$ has a dicritical singularity with
non-nilpotent linear part. \item[\rm(4.2)] $X$ exhibits a
non-nilpotent singularity having at least three separatrices.
\item[\rm(4.3)] $X$ has some  singularity  in the Poincar\'e
domain and $\fa(X)$ has a meromorphic first integral in $N$, for
instance if $X$ has some simple  dicritical singularity.
\end{itemize}

\end{itemize}
}
\end{Definition}

 The following result is a mix of  characterization of connected Stein surfaces
biholomorphic to $\bc^2$ and partial classification of holomorphic
flows on Stein surfaces. It may be also seen as a follow-up to
results in \cite{[Ito-Scardua 4],Scarduacomplexvariables, Suzuki1} and
\cite{Suzuki2}:

\begin{Theorem}
\label{Theorem:equivalences}  Let $N$ be a connected Stein surface
with $H^2(N,\mathbb Z)=0$. The following conditions are
equivalent:
\begin{itemize}\itemsep=-1pt
\item[\rm(1)] $N$ is biholomorphic to $\bc^2$.
\item[\rm(2)] $N$ admits a complete holomorphic vector field $X$ with
a finite number of singularities and satisfying one of the following conditions:
\begin{itemize}\itemsep=-1pt
\item[\rm(a)] $X$ has the global transversality property.
\item[\rm(b)] $X$ has the first jet resonance property. \item[\rm(c)]
$X$ has the dicriticalness property. \item[\rm(d)] $X$ has the
stability property.
\end{itemize}
\end{itemize}

\end{Theorem}
Theorems~\ref{Theorem:main} and \ref{Theorem:equivalences} are
proved in Section~\ref{section:mainresult} and
Section~\ref{Section:equivalences} respectively. Some undemanding
consequences of Theorem~\ref{Theorem:equivalences} are:

\begin{Corollary}
Let $N$ be a connected Stein surface with $H^2(N,\mathbb Z)=0$
equipped with a complete holomorphic vector field $X$ with
finitely many singularities in $N$ and having a non-nilpotent
dicritical singularity, e.g., if $X$ has a singularity where its
first jet is of the form $X_\lambda$ for $\lambda \in \mathbb
Q_+$. Then $N$ is biholomorphic to $\bc^2$ and $X$ is globally
analytically linearizable in $N$.
\end{Corollary}

\begin{Corollary}
[\cite{Scarduacomplexvariables}]
Let $X$ be a complete holomorphic vector field on $\bc^2$. If $X$
has some simple dicritical singularity then $X$ is globally
linearizable in $\bc^2$ as $X_\lambda$ for some $\lambda\in\mathbb
Q_+$.
\end{Corollary}

\begin{Remark}
{\rm The conclusion of Theorem~\ref{Theorem:main} is
clearly false if we do not assume that $X$ has the first jet resonance
property. For instance take the vector field
$$X=\lambda
_1x\,\frac{\po}{\po x} + \la_2y\,\frac{\po}{\po y} + P(x^p y^q)
(qx\,\frac{\po}{\po x} -
  py\,\frac{\po}{\po y}),$$ where  $\lambda_1,\lambda_2 \in \bc^*$,
  \,\,   $P$ is a polynomial in one variable and  $p,q \in \bn$, \,\, $\lg p,q\rg = 1$.
Then $X$ is complete in $\bc^2$ but is not analytically conjugate
on $\bc^2$ to a vector field $X_\lambda$ or $Y_n$.}
\end{Remark}

Section~\ref{section:stability} contains some Global Stability
results for holomorphic flows on connected Stein surfaces. As a
model we have:

\begin{Theorem}
[Stability theorem, \cite{Scarduacomplexvariables}]
 \label{Theorem:stability}  Let $X$ be a {\rm complete}
holomorphic vector field on a connected Stein surface $N$ with
$H^2(N,\bz) = 0$. Suppose that $X$ has some closed orbit $L_0$
biholomorphic to $\bc^* = \bc-\{0\}$ having finite holonomy group.
Then $\fa(X)$ admits a meromorphic first integral and, in
particular, all nonsingular orbits are closed off $\sing(\fa(X))$
and have finite holonomy groups.
\end{Theorem}

Our basic models $X_\lambda$ and $Y_n$ appearing in
Theorem~\ref{Theorem:main} are studied in the sequel:

\begin{Example} \label{Example 5.1} {\rm Let $X_\la = x\,\frac{\po}{\po x} + \la
y\,\frac{\po}{\po y}$\,, $\la \in \bc\backslash\re_-$ and $Y_n :=
y\,\frac{\po}{\po y} + (nx+y^n)\,\frac{\po}{\po x}$ where $n \in
\bn$. Then $X_\la$ and $Y_n$ are complete holomorphic vector
fields on $\bc^2$ and also they generated foliations $\fa(X_\la)$
and $\fa(Y_n)$ with the following properties:
\begin{itemize}\itemsep=-1pt
\item[(i)] $\fa(X_\lambda)$ is transverse to the spheres $S^3(r)$,
$\forall\,r > 0$. \item[(ii)] $\fa(Y_n)$ is transverse to the
sphere $S^3(r)$ if $r > 0$ is small enough and if $n=1$ then
$\fa(Y_1)$ is transverse to all sphere $S^3(r)$, $r > 0$.
\end{itemize}}
\end{Example}
In both cases the foliation has a single singularity on $\bc^2$
and this singularity is either nondicritical (for $Y_n$ or $X_\la$
with $\la \notin \bq_+$) or the singularity (and also the globally
defined foliation) admits a meromorphic first integral of type
$\frac{x^k}{y^\ell}$\,, $k,\ell \in \bn$ (for $X_\la$ in the case
$\la = \frac{k}{\ell}$).


We recall that the basin of attraction of the origin as a
singularity of $X_\la$ is the whole plane $\bc^2$,  this is clear
from flow integration. The same holds for $Y_n$\,, a holonomy
argument for the case $n=1$ is as follows. It is well-known that
the holonomy map of the leaf $L = \{y=0\}\backslash\{0\}$ is a
homography of the form $h(y)=\frac{y}{1+ay}$\,, $a \in
\bc\backslash\{0\}$; this holonomy map is defined on a transverse
section $\Sigma$ of the form $\Sigma : \{x=1\}$. Thus, since the
orbits of $h$ accumulate the origin of $\Sigma$, every leaf of
$\fa(Y_n)$ accumulates the leaf $L$ which, by its turn,
accumulates the origin. Another general argument comes from the
integration of the Pfaffian equation: $\Om = y\,dx - (nx+y^n)dy =
0$. We observe that $\Om=-y^{n+1}d\ln (y e^{-\frac{x}{y^n}})$ and
we have a Liouvillian first integral  $F = ye^{-\frac{x}{y^n}}$
whose fibers clearly accumulate $\{y=0\}$ and therefore the origin
$0\in\bc^2$.


Let $X$ be a vector field on $N$ and let $\xi\subset N\setminus \sing(X)$ be a subset of $N$ consisting of orbits of $X$. We shall say that a point $q\in N$ is a {\it convergence point} for $\xi$ if each neighborhood $V$ of $q$ in $N$ intersects infinitely many orbits in $\xi$. In this case the orbit through $q$ is also contained in the set of convergence points of $\xi$. We shall say that $\xi$ is a {\it divergent} subset of $N$ if its set of convergence points is empty.

\begin{Theorem}
\label{Theorem:flowsstability}
Let $X$ be a {\rm complete} holomorphic vector field with isolated singularities on a connected Stein surface $N$ with $H^2(N,\bz) = 0$. Denote by $\xi(X)\subset N$ the subset of periodic orbits, i.e., orbits diffeomorphic to $\mathbb C^*$. Then there are the following possibilities:
\begin{enumerate}
\item The set of periodic  orbits has no convergence set on $N$, or  contains only simply-connected orbits in the convergence set.
 \item There is a meromorphic first integral for the flow on $N$.
  \item We have a (nRnL) periodic orbit.
 \end{enumerate}
\end{Theorem}

We shall explain in \S 3 what we mean by a (nRnL) orbit (see Definition~\ref{Definition:nRnLorbit}). As for the moment it suffices to that this is an orbit, diffeomorphic to $\mathbb C^*$, whose holonomy map is  a non-resonant and not analytically linearizable ( i.e., a so called ({\it nRnL})) map.

In the last section we study flows on affine (algebraic) surfaces. 
Our main result for this situation is:

\begin{Theorem}
\label{Theorem:affinesurface}
Let $X$ be an algebraic vector field on a regular affine surface $N\subset \mathbb C^m$ and assume that $X$ is complete in $N$ and the projectivization $\ov{N}\subset \mathbb CP^m$ is such that:
\begin{enumerate}
\item The singularities of $\ov{\fa}$ in $D=\ov{N}\setminus N$ are generalized curves.
\item The flow of $X$ has infinitely many periodic orbits  in $N$.
\end{enumerate}
Then,  there is a rational first integral for $\ov{\fa}$, there is a meromorphic first integral for $\fa$ in $N$, or $\ov{\fa}$ is given by a simple poles closed rational one-form $\Omega$ in $\ov N$.
\end{Theorem}

The exact meaning of the hypotheses is given in \S 6. The classification of such flows is then finished by means of Remark~\ref{Remark:complements}, where we deal with the case of only finitely many periodic orbits. Indeed, the statement of Theorem~\ref{Theorem:affinesurface} can be completed to the case of only finitely many periodic orbits. Also it can be completed to assure the existence of a rational first integral or a closed rational one-form, always that we have a meromorphic first integral on $N$ (see Remark~\ref{Remark:complements}).

\section{Preliminaries}

\subsection{Holomorphic flows on Stein spaces}

We consider an analytic action $\vr:\mathbb C\times V\to V$ on a
  normal Stein variety (the general theory of such spaces can be found
  in \cite{[Gunning-Rossi]} and
  \cite{Gunning III}). This means that $\vr(0,z)=z$ and
  $\vr(s_1,\vr(s_2,z))=\vr(s_1+s_2,z)$ for any $z\in V$ and
  $s_1, s_2\in\mathbb C$. The orbit of $z\in V$ is the subset
  $ O_z:=\{\vr(s,z);s\in\mathbb C\}\subset V$. The isotropy subgroup
  of $z\in V$ is $ G_z:=\{s\in \mathbb C; \vr(s,z)=z\}$. The map
  $s\mapsto \vr_{s} (z)=\vr(s,z)$ induces an injective immersion of
  $\mathbb C/G_z$ in $V$ whose image is $O_z$. Thus the orbits are
  either isomorphic to $\mathbb C$, $ \mathbb C^*$, a torus, or a
  point. Since $V$ is Stein it cannot contain compact analytic subsets,
  therefore no orbit can be a torus . A point $z\in V$ is called
  \emph{ fixed} if $\vr (s,z)=z$ for any $s\in \mathbb C$. Clearly its
  orbit reduces to $\{z\}$. Denote by ${\rm Fix}(\vr )$ the set of fixed points
  of $\vr$. We will assume that the set ${\rm Fix}(\vr)$ is discrete. Then,
  the orbits of $\vr$ are the leaves of a one dimensional foliation
  $\fa_{\vr}$ with a discrete set of singularities, $\sing\fa_{\vr}={\rm Fix}(\vr)$.

  Any action $\psi=(\psi_t)_{t\in \mathbb C^*}$ of the group $\mathbb C^*$
  on a Stein space $V$ induces an action $\vr=(\vr_{s})_{s\in \mathbb C}$
  of $\mathbb C$ on $V$ given by $\vr_{s}=\psi_{\exp(s)}$.\\
The study of the holomorphic flows  on  Stein spaces of dimension two
begins with M. Suzuki in  \cite{Suzuki1} and \cite{Suzuki2}. He
proved that if the generic orbit of $\vr$ is isomorphic
to $\mathbb C^*$, then the foliation $\fa_{\vr}$ admits a meromorphic first
integral.
Using the theory, developed by Nishino and Saito, of entire and meromorphic
functions on $\bc^2$ according to their prime surfaces, Suzuki proved analytic
linearization theorems for holomorphic actions of $\bc^*$ on
$\bc^2$ . He also gave a good description of the possible normal
forms of a holomorphic flow with proper orbits on $\bc^2$.
\subsection{Suzuki's theory}
\label{section:suzukitheory}

A fundamental contribution to the study of holomorphic flows and
foliations on  Stein surfaces, was made by M. Suzuki who
introduced on this subject the use of techniques of Potential
Theory and the Theory of Analytic Spaces (cf. \cite{Suzuki1} and
\cite{Suzuki2}).

Let us recall some of Suzuki's results.
\vglue.1in
  \noindent{\bf Theorem of Suzuki, \cite{Suzuki1}}:
{\it 
  Given a $\mathbb C$- action $\vr$ on
  a normal Stein analytic space $V$ of dimension $n \geq 2$:

  \begin{itemize}

  \item[{\rm(i)}] There is a subset $e\subset V$ of logarithmic capacity
  zero such that $\vr_{t}(e)=e$, for any $t\in \mathbb C$, and all orbits
  of $\vr$ in $V\setminus e$ are of the same topological type.

  \item[{\rm(ii)}] Any leaf of $\fa_{\vr}$ containing an orbit of $\vr$
  isomorphic to $\mathbb C^*$ is closed in $V\backslash sing {\fa_\vr}$.

\item[{\rm(iii)}] If $n=2$ and the leaves of $\fa_\vr$ are
properly embedded in $V\setminus \sing(\fa_\vr)$ then there is a
  meromorphic first integral of $\fa_{\vr}$ on $V$, not constant, and one
  can find a Riemann surface $S$ and a surjective holomorphic map
  $p:V\backslash sing { \fa_{\vr}}\to S$, such that
  1. The irreducible components of the fibers $\{p^{-1}(w); w\in S\}$ of $p$
  are the leaves of $\fa_{\vr}$.\\
  2. The subset $e$ union of all the nonirreducible levels $p^{-1}(w)$; $w\in S$,
   has zero logarithmic capacity.

\item[{\rm(iv)}] If $n=2$ and the generic leaf is isomorphic to $\mathbb C^*$,
then any leaf of $\fa_{\vr}$ is closed in $V\backslash sing { \fa_{\vr}}$ and
therefore there is a meromorphic first integral as in {\rm(iii)}.

  \end{itemize}
}

\vglue.1in

\subsection{Non-resonant map germs and singularities}

 A germ of a complex diffeomorphism $f$ at the origin $0 \in \mathbb C$ writes
$f(z)= e^{2 \pi \sqrt{-1} \lambda} z + a_{k+1} z^{k+1} + ...$.  If $\lambda \notin \mathbb R$ then $|f^\prime(0)| \ne 1$, the germ is {\it hyperbolic},  it is {\it analytically linearizable} (\cite{Dulac}) and has no closed orbit outside the origin. If $f^\prime(0)$ is a root of the unity (i.e., if $\lambda \in \mathbb Q$) then the dynamics of $f$ is well-known (\cite{camacho}). In particular, none of the orbits outside the origin  is closed. If $f^\prime(0)\in S^1$ is not a root of the unity then we have $\lambda \in \mathbb R \setminus \mathbb Q$. In this case we shall say that the diffeomorphism is {\it non-resonant}. Such a map is formally but not always analytically linearizable (see  Cremer \cite{Cr1}). Such germs have been studied by several authors as Siegel, Brjuno, Yoccoz and P\'erez-Marco.
It is well-known that the map germ $f$ is analytically linearizable if, and only if, it is topologically conjugate to its linear part. This is also equivalent to the fact that the map $f$ is stable (\cite{raissy, bracci}). Assume now that the map is not analytically linearizable. Then, a remarkable result of P\'erez-Marco (\cite{perezmarco}) shows that under certain diophantine conditions on the linear part  of the diffeomorphism, we have either the diffeomorphism is analytically linearizable or admits the  existence of {\it small cycles}, that is, periodic orbits contained in any neighborhood of the origin. On the other hand, since these maps can be constructed non-linearizable, we conclude that there exist non-resonant diffeomorphism germs $f$ exhibiting small cycles but, having in any neighborhood of the origin, some
non-periodic orbit.

A very nice description of the dynamics, including the existence of the so called {\it hedgehogs}  of such maps is given by P\'erez-Marco in \cite{P6, P7}.

On the other hand, Yoccoz and P\'erez-Marco proved that
given any non-resonant map germ $f$, there is a germ of a holomorphic foliation $\fa(f)$ at the origin $0\in \mathbb C^2$, such that:
\begin{itemize}
\item $\fa(f)$ is a germ of isolated singularity in the Siegel domain.
\item The holonomy of one of the two separatrices of $\fa(f)$ is analytically conjugate to $f$.
    \end{itemize}

In particular, in suitable local coordinates we can write $\fa(f)$ as given by
$x(1+ A(x,y)) dy - \lambda y(1 + B(x,y))dx=0$, for some holomorphic $A(x,y), \, B(x,y)$ with $A(0,0)=B(0,0)=0$. In the normal form above, the separatrices are the coordinate axes. From the above properties for $f$ the foliation $\fa(f)$ exhibits the following characteristics:

\begin{itemize}
\item $\fa(f)$ is in the Siegel domain and is a non-resonant singularity.
\item $\fa(f)$ has on each neighborhood of the origin, some closed leaf and also some leaf which is not closed and accumulates on one of the separatrices.
\item $\fa(f)$ is not analytically (nor topologically) linearizable.

\end{itemize}
In particular, $\fa(f)$ is not analytically linearizable, but admits infinitely many closed orbits. Also $\fa(f)$ always admits some leaf which accumulates at the separatrices, but not only at the separatrices.

On the other hand, since any non-resonant diffeomorphism is  formally linearizable, a foliation as above is formally similar to the linear model $xdy - \lambda ydx=0$.

\begin{Definition}
{\rm A non-resonant and not analytically linearizable map germ $f\in \Diff(\mathbb C,0)$ will be called a {\it {\rm(}nRnL{\rm)} map}. Similarly, a foliation in the Siegel domain whose holonomy is represented by a (nRnL) map is called a  {\it {\rm(}nRnL{\rm)} foliation singularity}.
}
\end{Definition}

In this paper we   classify  the groups of germs at a fixed point of complex diffeomorphisms in the complex line that exhibit infinitely many closed orbits. It turns out that they are either finite groups or generated by a (nRnL) map and  a rational rotation (cf.Proposition~\ref{Proposition:finiteorbitsnonresonant}).

\section{Groups with many closed orbits}{}
Let  $\Diff({\mathbb C},0)$  denote the group of germs at the origin $0\in {\mathbb C}$of holomorphic diffeomorphisms. It is a well-known result that a a finite group of germs of complex diffeomorphisms is analytically conjugate to a cyclic group generated by a rational rotation. Thus such a finite group is, up to a holomorphic change of coordinates, of the form  $\{\exp(2 k\pi \sqrt{-1}/\nu), \, k=0,1,...,\nu-1\}$ for some $\nu \in \mathbb N$.
Now we shall extend this fact in what follows.
\begin{Definition}[resonant group]
\label{Definition:resonantgroup}
{\rm A germ of a complex diffeomorphism $g \in \Diff(\mathbb C,0)$ is called
{\it resonant} if its linear part is a root of unity, i.e., $g^\prime(0)= \exp(2 \pi \sqrt{-1} k /\ell)$ for some $k , \ell \in \mathbb N$.
A group $G\subset \Diff(\mathbb C,0)$ of germs of holomorphic diffeomorphisms
will be called {\it resonant} if each map $g \in G$ is a resonant germ. This is equivalent to the fact that $G$ has a set of generators consisting only of resonant maps.
}
\end{Definition}
The
 following  is, for the case of resonant groups,  a  generalization of a result in \cite{mattei-moussu}.

\begin{Lemma}
\label{Lemma:finiteorbits} Let $G\subset \Diff({\mathbb C},0)$ be a finitely generated resonant subgroup such that the set $\mathcal C(G)$ of points having closed pseudo-orbit is infinite in a neighborhood $U$ of the origin $0\in {\mathbb C}$. Then $G$ is finite cyclic and analytically conjugate to its linear part.
\end{Lemma}

\noindent{\bf Proof}. First we observe that by Nakai  density theorem (\cite{nakai}) the group $G$ must be solvable, otherwise all orbits are dense except for some finite  set of real  analytic invariant curves. Nevertheless, also according to \cite{nakai}, the orbits in these curves are also dense in these curves and therefore not finite.
On the other hand, if $G$ is not an abelian group, then a non-trivial element $g\in G$ in the commutator of $G$, is of the form $g(z)= z + a_{k+1} z^{k+1} + hot.,$ $a _{k+1} \ne 0$. According to \cite{camacho} the pseudo-orbits outside the origin of $g\in G$ are not finite. Thus $G$ must be abelian  and cannot contain elements of the form $g(z)= z + a_{k+1} z^{k+1} + hot.,$ $g\ne \Id$. We claim that any element $g\in G$ has finite order. Indeed, assume that there exists $g\in G$ with $g^n \ne\Id$ for any $n\in \mathbb Z - \{0\}$. Because of what we observed above, we may assume that  $g(z)=\lambda \cdot z + hot. \, , \, \lambda^n \ne 1$ for any $n\in \mathbb Z-\{0\}$. This is not possible since the elements in $G$ are resonant.


\indent Since $G$ is abelian and finitely generated, the claim implies that $G$  itself is finite, and therefore $G$ must be a group of rational rotations up to analytic conjugation.
\qed

\begin{Remark}[solvable groups]

{\rm

For each $k\in \mathbb N$ we define the group

\[
\mathbb H_k = \left \{\vr \in \text{Diff}(\mathbb C,0); \, \,  \vr(z)^k =
\frac{\mu_\vr z^k}{1+a_\vr z^k}\,, \,\,\mu_\vr \in \mathbb C^*, a_\vr
\in \mathbb C \right \}.
\]

\noindent Then $\mathbb H_k$ is a  solvable group and, up to formal
conjugacy, any solvable nonabelian subgroup $G$ of
$\Diff({\mathbb C},0)$ is of this type (\cite{cerveau-moussu}, a result by
Cerveau-Moussu). Moreover, if the group of commutators $[G,G]$ is not cyclic,  this conjugacy is analytic.
}
\end{Remark}

We can complete the above remark as follows (\cite{scarduaJDCS}):
\begin{Lemma}
\label{Lemma:linearization} Let $G < \Diff(\mathbb C,0)$ be a solvable non-abelian
subgroup of germs of holomorphic diffeomorphisms fixing the origin
$0 \in \mathbb C$.
\begin{itemize}
\item[\rm(i)] If  the group of commutators
$[G,G]$ is not cyclic then $G$ is analytically conjugate to a
subgroup of $\mathbb H_k = \big\{z \mapsto
\frac{az}{\sqrt[k]{1+bz^k}}\big\}$ for some $k \in \mathbb N$.
\item[\rm(ii)] If there is some $f \in G$  of the form $f(z) = e^{2\pi
i\lambda}\,z +\dots$ with $\lambda \in{\mathbb  C}\backslash \mathbb Q$ then $f$ is
analytically linearizable in a coordinate that also embeds $G$ in
$\mathbb H_k$.
\end{itemize}

\end{Lemma}

 Another  simple remark is the following:
\begin{Lemma}
\label{Lemma:perezmarcocyclic}
Let $f\in \Diff(\mathbb C,0)$ be a map germ with linear part
$f^\prime(0)= e^{2 \pi \sqrt{-1} \lambda}$ where $\lambda \in \mathbb R \setminus \mathbb Q$. If some iterate of $f$ belongs to some cyclic group generated by a (nRnL) map germ, then $f$ is a (nRnL) map germ.
\end{Lemma}
\begin{proof}
By the hypothesis, there is some $k \in \mathbb N$ such that the $k$-iterate
$f^{(k)}$ is an iterate say, $g ^{(\ell)}$ of some (nRnL) map $g$, where $\ell \in \mathbb Z$. If  $f$ is not  a (nRnL) map then $f$ is analytically linearizable and the same holds for $f^{(k)}= g^{(\ell)}$. Since $g$ and its iterates commute, and since $g^{(\ell)}$ has a non-periodic linear part,  a simple computation with power series shows that $g$ is analytically linearizable in the same coordinate, absurd.
\end{proof}

The main fact regarding non fully resonant groups we will refer to is the following:

\begin{Proposition}
\label{Proposition:finiteorbitsnonresonant} Let $G\subset \Diff({\mathbb C},0)$ be a finitely
generated  subgroup such that the set $\mathcal C(G)$ of points having closed pseudo-orbit is infinite in a neighborhood $U$ of the origin $0\in {\mathbb C}$. Then we have two possibilities:

\begin{enumerate}
\item $G$ is finite cyclic and analytically conjugate to its linear part.
\item $G$ is abelian,  formally but not analytically linearizable, containing only periodic and  (nRnL) maps.
\end{enumerate}

\end{Proposition}

\begin{proof}
By Lemma~\ref{Lemma:finiteorbits} if $G$ is  a resonant group then it is cyclic as in (1). Let us then assume that $G$ is not a resonant group. Therefore, let $f\in G$ be a non-resonant diffeomorphism, with linear part $f^\prime(0)=e^{2 \pi \sqrt{-1}\lambda}$  where $\lambda\in \mathbb R \setminus \mathbb Q$.
As in the proof of Lemma~\ref{Lemma:finiteorbits}, first observe that $G$ is solvable by Nakai Density theorem (\cite{nakai}). We have two cases to consider. If $G$ is non-abelian then by a theorem of Cerveau-Moussu m(see \cite{scarduabounded}) the group $G$ is analytically conjugate to a subgroup of $\mathbb H_k$ for some $k \in \mathbb N$ (notice that $G$ contains a non-resonant germ, so that the formal conjugation is actually analytic). Now, the non-resonant map $f$ is analytically linearizable as a consequence of (ii) in Lemma~\ref{Lemma:linearization}. Therefore, the orbits of $f$ outside the origin are not closed. This is a contradiction. Thus we conclude that necessarily  $G$ is abelian. Because the non-resonant map $f$ is formally linearizable, an easy computation with formal power series shows that indeed the group $G$ is {\it formally  linearizable}, i.e, conjugate by a formal diffeomorphism to its linear part. As above, the fact that $G$ has closed orbits (arbitrarily) close to the origin implies that none of its non-resonant maps is analytically linearizable. In other words we have proved:

\begin{Claim}
Each map in $G$ is either a finite order map or it is a (nRnL) map.
\end{Claim}

Now let $G_{\res}\subset G$ be the subgroup generated by the resonant elements in $G$. By Lemma~\ref{Lemma:finiteorbits} $G_{\res}$ is cyclic finite, generated by a rational rotation.
Because $G$ is abelian, $G_{\res}$ is a normal subgroup of $G$. We may therefore use the above to prove that $G$ is the direct product of a group generated by a rational rotation and by a group generated by a (nRnL) map (\cite{camacho-scarduadarbouxlocal}).
\end{proof}

We end this section with the definition of (nRnL) orbit:
\begin{Definition}
\label{Definition:nRnLorbit}
{\rm
A periodic  orbit of a holomorphic flow will be called  a {\it
(nRnL) orbit} if its holonomy map is a (nRnL) map.
}
\end{Definition}

\section{Proof of Theorem~\ref{Theorem:main}}
\label{section:mainresult}

In this section we prove Theorem~\ref{Theorem:main}. In what
follows $X$ is a complete holomorphic vector field with isolated
singularities on a connected Stein surface $N$ satisfying
$H^2(N,\mathbb Z)=0$. We suppose that  $\fa(X)$  is transverse to
the boundary of a compact domain $D \subset N$ biholomorphic to
the closed unit $2$-disc  $\ov D^2 (1) \subset \bc^2$.

\begin{proof}[Proof of Theorem~\ref{Theorem:main}]  First we observe that
since $D$ is biholomorphic to the unit  $2$-disc  $\ov D^2 (1)$ in
$\bc^2$ we may apply \cite{[Ito]} to conclude that $X$ has a
single singularity $P_0 \in D\backslash\po D$ and each orbit which
crosses $\po D$ tends to $P_0$\, and  also $P_0$ must be a simple
singularity in the Poincar\'e domain. By Poincar\'e-Dulac theorem
we have two cases:

\vglue.1in \noindent{\bf $1^{\text{st}}$ Case}:\, $X$ is
analytically conjugate in a neighborhood $V$ of $P_0$ to
$X_\la$\,, $\la \notin \re_-$\,.

If $\la \in \bq_+$ then every leaf of $\fa(X)$ in $V$ contains a
separatrix of $X$ and we claim that such a leaf must be
biholomorphic to $\bc^* = \bc-\{0\}$. Indeed, a leaf $L$ of
$\fa(X)$ is a nonsingular orbit of the (complete) flow of $X$ and
therefore it is conformally equivalent to one of the following
Riemann surfaces: $\bc$, $\bc^*$ or a compact complex torus
$\bc/(\bz\oplus\bz)$\,. Since a Stein manifold contains no compact
curve the only possibilities are $L \cong \bc$ or $L \cong \bc^*$.
If $L \cong \bc$ then $\ov L \supset L \cup \{0\}$ and therefore
there is a holomorphic mapping $\ov\bc = \bc \cup \{\infty\} \to
\ov L \subset N$ which implies that $\ov L$ is compact again
yielding a contradiction. Therefore $L \cong \bc^*$. Now,
according to \cite{Suzuki1} the flow of $X$ does have a generic
orbit: there exists an invariant subset $\sigma \subset N$ with
zero transverse logarithmic capacity, such that the orbits of $X$
in $N-\sigma$ are pairwise biholomorphic. Clearly $(N-\sigma) \cap
V$ contains points in generic orbits and therefore the generic
orbit of $X$ is conformally equivalent to $\bc^*$. This implies by
\cite{Suzuki1} that the flow of $X$ admits a nonconstant
meromorphic first integral on $N$.

Regarding the flow linearization on $D$ we consider $\la \in
\bc\backslash\re_-$ and observe that, as remarked above, the basin
of attraction $B_{P_0}$ of $P_0$ contains the domain $D$. The flow
of $X$ gives a holomorphic conjugation between $X$ and $X_\la$ in
the basin $B_{P_0}$\,, we can therefore take $U = B_{P_0}$ and
$\Phi(U) = \bc^2$.

\vglue.1in

\noindent{\bf $2^{\text{nd}}$ Case}:\, $X$ is analytically
conjugate in a neighborhood $V$ of $P_0$ to $Y_n$\,, $n \in \bn$.

Again the flow gives an extension of the conjugation in $V$ to the
basin $B_{P_0}$ and since the basin of $Y_n$ is $\bc^2$ we have $U
= B_{P_0}$ and $\phi(U) = \bc^2$.

\vglue.1in
\par Finally we observe that since $P_0 \in D \subset B_{P_0}$ we
have $B_{P_0} = N$ if and only if every leaf of $\fa(X)$ crosses
$\po D$. This proves the theorem.\end{proof}

\section{Proof of Theorem~\ref{Theorem:equivalences}}
\label{Section:equivalences} Let us pave the way to the proof of
Theorem~\ref{Theorem:equivalences}. The argumentation in
\ref{subsubsection:Dicriticalness} and
\ref{subsubsection:Nondegeneracy} below  is partially inspired in
\cite{[Cerveau-Scardua]}.   Therefore, {\sl in what follows $X$ is
a complete holomorphic vector field on a connected Stein surface
$N$ with $H^2(N,\bz) = 0$}. We shall preserve the same notation
and will refer to the the proof of Theorem~\ref{Theorem:main}.

\subsection{Transversality}
\label{subsubsection:Transversality}
 We first assume that $X$ satisfies the
transversality condition (2.i), i.e., $X$ is transverse to $\po
D_j$\,, $\forall\,j$ where $\bigcup \limits_{j\in \mathbb N} D_j =
N$ is an exhaustion of $N$ by domains biholomorphic to the closed
$2$-disc $\ov D^2(1) \subset \bc^2$. By Theorem~\ref{Theorem:main}
given any index $j\in \mathbb N$ we have $\sing(X)\cap D_j
=\{P_j\}$ and the singularity $P_j$ is in the Poincar\'e domain,
moreover the flow of $X$ in $D_j$ gives a  holomorphic conjugation
$\Phi_j$ of $X\big|_{B_{P_j}}$ with  a  vector field  $Z_j$ which
is either linear $Z_j=X_{\lambda_j} = x\,\frac{\po}{\po x} +
\lambda_j y\,\frac{\po}{\po y}$\,, $\lambda_j \in
\bc\backslash\re_-$\,, or is  a  Poincar\'e-Dulac normal form
vector field $Z_j=Y_{n_j} = y\,\frac{\po}{\po y} +
(n_jx+y^{n_j})\,\frac{\po}{\po x}$\,, $n_j \in
\bn\backslash\{1\}$. The basin of attraction $B_{P_j}$ equals to
$N$ if and only if every leaf of $\fa(X)$ crosses the boundary
$\po D_{j}$. We claim that this is the case for each $j\in\mathbb
N$. Indeed, since $\bigcup \limits_{j\in \mathbb N} D_j = N$ is an
exhaustion of $N$ we must have $P_{j_1}=P_{j_2}, \, \,
Z_{j_1}=Z_{j_2}$ and also $B(P_{j_1})=B(P_{j_2}), \,  \forall
j_1,j_2 \in \mathbb N$. On the other hand, since $D_j\subset
B_{P_j}$ and $\bigcup \limits_{j\in \mathbb N} D_j = N$ we
conclude $\bigcup\limits_{j\in\mathbb N} B_{P_j} = N$ and from the
above remark $B_{P_j}=N$ for every $j \in \mathbb N$. Notice that
this already proves that (2.i) implies (1) in
Theorem~\ref{Theorem:equivalences}.

Now we assume that $X$ satisfies transversality condition (2.ii).
Then by Theorem~\ref{Theorem:main} we conclude that the
conjugation $\Phi$ between $X$ and one of the models $X_\lambda$
and $Y_n$ extends to all the manifold $N$ and therefore $\Phi$ is
a biholomorphism from $N$ onto $\bc^2$. This proves that (2.ii)
implies (1) in Theorem~\ref{Theorem:equivalences}.

\subsection{Dicriticalness}
\label{subsubsection:Dicriticalness}

\noindent We begin with a basic remark. Suppose that
  $X$ has some dicritical singularity in $N$.
By the argumentation in the first case in the proof of
Theorem~\ref{Theorem:main}, the generic orbit of $X$ is $\bc^*$
and $\fa(X)$ has a meromorphic first integral $\psi$ in $N$. Now
we proceed our study case by case.

First we assume that $X$ has a singularity $P_0\in \sing(X)$ such
that the first jet of $X$ at $P_0$ is a nondegenerate singularity
in the Poincar\'e domain. Suppose that $\fa(X)$ has a meromorphic
first integral (e.g., if $X$ has a dicritical singularity in $N$).
Since $\fa(X)$ has a meromorphic first integral in $N$ the
singular point $P_0$ also admits a meromorphic first integral and
we conclude from Poincar\'e-Dulac normal form theorem
\cite{Arnold} that necessarily the vector field $X$ is
linearizable in a neighborhood of $P_0$ conjugate to a vector
field of the form $X_\lambda$ where $\lambda= \frac{k}{\ell} \in
\bq_+$ (indeed, the nonlinear normal form $Y_n$ admits no local
meromorphic first integral because the holonomy of the invariant
manifold $\{y=0\}$ is not finite since it is nontrivial and
tangent to the identity).

Assume now that $X$ has at a singularity $P_0\in \sing(X)$ a first
jet of the form $j_{P_0}^1\,X = \la_1x\,\frac{\po}{\po x} +
\la_2y\,\frac{\po}{\po y}$ with $\frac{\la_1}{\la_2} =
\frac{k}{\ell} \in \bq_+$. In this case  $P_0$ is a dicritical
singularity  and therefore $\fa(X)$ has a meromorphic first
integral on $N$. From the above $X$ is linearizable in a
neighborhood of $P_0$, indeed in the basin of attraction
$B_{P_0}$.

Finally, we assume that   $X$ is analytically conjugate to a
vector field  $X_\la$ with $\la \in \bq_+$ in a neighborhood of a
singular point $P_0$. Then,  necessarily $X$ is conjugate to
$X_\la$  in $B_{P_0}$ and therefore  we can write $X(x,y) =
x\,\frac{\po}{\po x} + \frac{k}{\ell}\,y\,\frac{\po}{\po y}$\,,
$k,\ell \in \bn$ in suitable holomorphic coordinates in $B_{P_0}$.
Since $N$ is a Stein space with $H^2(N,\bz) = 0$ we can write the
meromorphic function $\psi$ as $\psi = \frac{f_1}{f_2}$ for some
holomorphic function $f_j\colon N \to \bc$ and such that $f_1$\,,
$f_2$ are locally relatively prime. By the local form of $X =
x\,\frac{\po}{\po x} + \frac{k}{\ell}\,y\,\frac{\po}{\po y}$ we
conclude that $f$ is indeed of the form $f = \frac{\tilde
f_1^k}{\tilde f_2^{\ell}}$ for some reduced holomorphic functions
$\tilde f_j \colon N \to \bc$, $j=1,2$ such that $\{f_1=0\} \cap V
= \{x=0\}$ and $\{f_2=0\} \cap V = \{y=0\}$.

\begin{Lemma} \label{Lemma:boundary} Assume that $X$ has some singularity $P_0$ which is
nondegenerate and in the Poincar\'e domain and suppose that $X$
has some dicritical singularity in $N$ or, more generally, that
$\fa(X)$ has a meromorphic first integral in $N$. Then the
boundary $\po B_{P_0}$ is a
 {\rm(}possibly empty{\rm)} union of isolated nondicritical singularities of $X$
and analytic curves, each curve accumulates a unique nondicritical
singularity of $X$.
\end{Lemma}

\begin{proof} Since $\po B_{P_0}$ is invariant by $\fa(X)$
if it is nonempty then it is a union of leaves of $\fa(X)$ and
singular points of $X$. We divide the argumentation in several
steps.

\vglue.1in \noindent{\bf Step 1}:\, $\po B_{P_0}$ contains no
closed leaf biholomorphic to $\bc^*$.
\begin{proof}[Proof of Step 1]
Let $L_0 \subset \po B_{P_0}$ be a leaf of $\fa(X)$ biholomorphic
to $\bc^*$. Since $\fa(X)$ admits a meromorphic first integral,
either $L_0$ is closed in $N$ or it accumulates only on singular
points, i.e., $\ov L_0 \subset L_0 \cup \sing (\fa(X))$. Suppose
$L_0$ is closed in $N$ then it is an analytic smooth curve in $N$.
Since $N$ is Stein and $H^2(N,\bz)= 0$ we can obtain a reduced
equation $\{h=0\}$ for $L_0$ in $N$, where $h\colon N \to \bc$ is
holomorphic. Since $L_0$ is a real surface diffeomorphic to a
cylinder $S^1 \times \re$ by a result of Richards
\cite{[Richards]} we can take a generator $\ga\colon S^1 \to L_0$
of the homology of $L_0$ and a holomorphic one-form $\al$ on $L_0$
such that $\displaystyle\int_\ga \al = 1$. Since $N$ is Stein with
$H^2(N,\bz) = 0$ there is a holomorphic one-form $\tilde\al$ on
$N$ which extends $\al$. Because the holonomy of $L_0$ is finite
(recall that $\fa(X)$ has a meromorphic first integral on $N$)
there is a fixed power of $\ga$ which has closed lifts
$\tilde\ga_z$ to the leaves $L_z$ of $\fa(X)$ that contain the
points $z \in \Sigma$, where $\Sigma$ is a sufficiently small
transverse disc to $\fa(X)$ with $\Sigma \cap L_0 = \{p_0\} \in
\ga(S^1)$.  Thus, for $z \in \Sigma$ close enough to $p_0$ we have
$\big|\displaystyle\int_{\tilde\ga_z} \tilde\al -
\displaystyle\int_{\tilde\ga_{p_0}}\tilde \al \big| < \frac 12$\,,
but $\tilde\ga_{p_0} = \ga$ and, since $\ga \subset L_0$\,,
$\displaystyle\int_{\tilde\ga_{p_0}} \tilde\al =
\displaystyle\int_\ga \tilde\al\big|_{L_0} = \displaystyle\int_\ga
\al = 1$ so that $\displaystyle\int_{\tilde\ga_z} \tilde\al \ne
0$. Or the other hand $\tilde\al$ is holomorphic to that
$\tilde\al\big|_{L_z}$ is holomorphic and therefore closed what
implies, since $\tilde\ga_z \subset L_z$ is closed, that $L_z$ has
nontrivial homology and therefore necessarily $L_z \cong \bc^*$.
Furthermore, every leaf intersecting $\Sigma$ (small enough
transverse disc intersecting $L_0$) cannot extend to (cannot embed
into) a holomorphic curve $S^1 \hookrightarrow \bc^2$ with trivial
homology. However, since $L_0 \subset \po B_{P_0}$ there are
leaves $L$ of $\fa(X)$ such that $L$ intersect discs $\Sigma$ as
above and which satisfy $L \subset B_{P_0}$\,. Such a leaf $L$
accumulates $P_0$ and therefore $L \cup \{P_0\}$ is a holomorphic
curve biholomorphic to $\bc$ and therefore with trivial homology;
yielding a contradiction.
\end{proof}

\noindent{\bf Step 2}:\, All leaves of $\fa(X)$ are biholomorphic
to $\bc^*$.
\begin{proof}[Proof of Step 2]
Indeed, in the open set $B_{P_0}$ the flow of $X$ is conjugate to
the periodic flow of $X_\la$\,, $\la = k/\ell$ so in $B_{P_0}$
there is a time $\tau \in \bc \backslash\{0\}$ such that the flow
$X^t$ of $X$ satisfies $X^{t+\tau} = X^t$, $\forall\, t \in \bc$.
By the Identity Principle ($N$ is always assumed to be connected)
we get that the flow of $X$ is periodic of period $\tau$ and
therefore all nonsingular orbits are biholomorphic to $\bc^*$.
\end{proof}

\noindent{\bf Step 3}:\, $\po B_{P_0}$ is a union of isolated
nondicritical singularities and analytic curves each of these
curves contains a unique nondicritical singularity of $X$.
\begin{proof}[Proof of Step 3]
In fact, by the first two steps each leaf $L$ contained in
$B_{P_0}$ must be biholomorphic to $\bc^*$ and is not closed in
$N$ so that it accumulates some singularity $P$ of $X$ and since
$\ov L \supset L \cup \{P\} \simeq \bc^* \cup \{0\} = \bc$ and
$\ov L$ cannot be compact, it follows that $\ov L = L \cup
\{P_0\}$ is an analytic curve in $N$ and $L$ accumulates no other
singularity of $X$.

\par Now we study the isolated points of $\po B_{P_0}$. Given an
isolated point $P \in \po B_{P_0}$ clearly $P \in \sing(X)$. If
$P$ is a dicritical singularity then, since $\fa(X)$ has a
meromorphic first integral, all leaves close enough to $P$
accumulate $P$ and this is not possible for leaves $L \subset
B_{P_0}$ because, as we have seen, a leaf cannot accumulate two
distinct singularities (such a leaf would be contained in a
rational curve in $N$ what is not possible because $N$ is Stein).
Thus $P$ is not dicritical and $\fa(X)$ has a holomorphic first
integral in a neighborhood of $P$.

Finally we observe that if a leaf $L \subset \po B_{P_0}$
accumulates at a  singularity $P \in \sing(X)$ then, as above, this
singularity is nondicritical: the basin of attraction of a
dicritical singularity is open and contains an open neighborhood
of the singularity so that, since $P \in \po B_{P_0}$\,, some leaf
$L_1 \subset B_{P_0}$ intersects this neighborhood and therefore
$L_1$ accumulates both $P_0$ and $P$ which is impossible. Thus $P$
is nondicritical. \end{proof}

 Lemma~\ref{Lemma:boundary} is proved as a consequence of the three steps above.\end{proof}

\begin{Lemma} \label{Lemma:biholomorphictoplane}  Let $X$ be a complete holomorphic
vector field on a connected Stein surface $N$ with $H^2(N,\mathbb
Z)=0$ and $P_0\in \sing(X)$ be a dicritical linearizable
singularity.  Under the above framework assume that the number of
analytic curves in $\po B_{P_0}$ is finite {\rm(}e.g., if
$\sing(X)$ is a finite set{\rm)}. Then $N = B_{P_0}$ and in
particular $N$ is biholomorphic to $\bc^2$.
\end{Lemma}

\begin{proof}  Let us first prove that $N=B_{P_0}\cup \po B_{P_0}$.
Indeed, put $A = N-\po B_{P_0}$ and $B =
B_{P_0}$\,. Then $A$ is a connected open subset of $N$ because
under the hypotheses above $\po B_{P_0}$ is a thin set and
therefore it does not disconnect $N$. Since $B$ is also open and
connected and $\po A = \po B$ it follows that $A=B$ because $B
\subset A$. Therefore $N = B_{P_0} \cup \po B_{P_0}$\, and
$N\setminus \po B_{P_0}=B_{P_0}$. Let us prove that
 $\po B_{P_0} = \emptyset$. Suppose
there is some analytic curve (leaf) $L_0 \subset \po B_{P_0}$\,.
We take a reduced equation $L_0 : \{f=0\}$ for $L_0$\,, given by
$f\colon N \to \bc$ holomorphic. Define then the meromorphic
one-form $\al = \frac{df}{f}$ on $N$. Given a transverse disc
$\Sigma \cong \bd$ to $\fa(X)$ with $\Sigma \cap L_0 = \{0\}$ we
consider a simple loop $\ga\colon S \to \Sigma$ around $p \in
\Sigma$ such that $\displaystyle\int_\ga \al = 2\pi\sqrt{-1}$.
 We can assume that $\ga(S^1) \subset B_{P_0}$ because
$\Sigma\setminus (\Sigma \cap \po B_{P_0})\subset B_{P_0}$ and
$\po B_{P_0}$ is thin. Since $B_{P_0}$ is simply-connected it
follows that $\ga$ is null-homotopic in $N\backslash\po B_{P_0} =
B_{P_0}$ and since $\al$ is closed and holomorphic in $B_{P_0}$ it
follows that $\displaystyle\int_\ga \al = 0$ yielding a
contradiction. This proves that $\po B_{P_0}$ must have dimension
zero, i.e., it is a (discrete) subset of $\sing(X)$. Now a similar
argumentation for an isolated point $p \in \po B_{P_0}$ shows
that, since $H_2(B_{P_0},\re) = 0$, we must have $\po B_{P_0} =
\emptyset$. Therefore $N = B_{P_0}$ and the lemma is
proved.\end{proof}

\subsection{First jet resonance}
\label{subsubsection:Nondegeneracy}
 In the proof of Theorem~\ref{Theorem:equivalences} we shall need the
 following lemma:

\begin{Lemma} \label{Lemma:firstjet}  Let $X$ be a {\rm complete} holomorphic
vector field on a connected Stein surface $N$ with $H^2(N,\mathbb
Z)=0$ and $P_0 \in \sing(X)$ an {\rm isolated} singularity of $X$.
Then the first jet of $X$ at $P_0$ is non zero.
\end{Lemma}

\begin{proof}  We basically follow \cite{[Rebelo]}. By the
Separatrix Theorem \cite{[Camacho-Sad]} there exists a germ of
invariant analytic curve $\Ga$ (possibly singular at $P_0$) with
$P_0 \in \Ga$. Let $L_0 \subset N$ be the (non-singular) orbit of
$X$ that contains $\Ga-\{P_0\}$. We have two possibilities:

\noindent{$1^{\text{st}}$}:\,  $L_0$ is biholomorphic to $\bc$.

This is not possible because $N \supset \ov L_0 \supset L_0 \cup
\{P_0\} = L_0 \cup \Ga \cong \ov \bc$ and $N$ contains no compact
curve.

\noindent{$2^{\text{nd}}$}:\,  $L_0$ is biholomorphic to $\bc^*$.

In this case $L_0$ is closed off $\sing(X)$ and therefore $\ov L_0
= L_0\cup \{P_0\} = L_0 \cup \Ga \cong \bc$. We have a global
``Puiseaux" parametrization $\sigma\colon \bc \to \ov L_0 \subset
N$ $($even if $\Ga$ is singular at $0)$ which is an injective
holomorphic map $\sigma$ with $\sigma(0) = P_0$ and $\sigma'(z)
\ne 0$, $\forall\,z \in \bc\setminus \{0\}$.
 Thus we can lift the vector field $X^* = X\big|_{\ov
L_0}$ to a holomorphic vector field $\widetilde X^*$ on $\bc$.
 The equation $\sigma_*(\widetilde X^*) = X^*$ implies
that $\sigma$ conjugates the flows of $\widetilde X^*$ and $X^*$
and since $X^*$ is complete with $X^*(P_0) = 0$ we conclude that
$\widetilde X^*$ is a complete holomorphic vector field on $\bc$
with $\widetilde X^*(0) = 0$. According to
\cite{[Cerveau-Scardua]} this implies that $\widetilde X^*(z) =
az$ for some $a \in \bc\backslash\}0\}$ because $\widetilde X^*
\not\equiv 0$.

Now we consider the Taylor development of $\sigma\colon \bc \to
\bc^2$ as
$$
\sigma(z) = \sigma^{(1)}(0)\bullet z +
\frac{1}{2!}\,\sigma^{(2)}(0)\bullet z^2 +\cdots+
\frac{1}{j!}\,\sigma^{(j)}(0)\bullet z^j +\dots
$$
Suppose by contradiction that $DX(P_0) = 0$ then $DX^*(P_0) = 0$
and we have from the conjugacy equation for the flows
 $X_t^*(\sigma(z)) = \sigma\big(\widetilde X_t^*(z)\big), \forall\,t \in \bc$ that
$$
X_t^*\big(\sigma^{(1)}(0)\bullet z + \sum_{j\ge2}
\frac{1}{j!}\,\sigma^{(j)}(0)\bullet z^j\big) =
\sigma^{(1)}(0)\bullet\widetilde X_t^*(z) + \sum_{j\ge2}
\frac{1}{j!}\, \sigma^{(j)}(0)\bullet \big(\widetilde
X_t^*(z)\big)^j \, , \,\forall\,t \in \bc.$$

\noindent Since $\sigma(z) \in \ov L_0$\,$\forall\,z$ and $X_t^* =
X_t\big|_{\ov L_0}$ and $\widetilde X_t^*(z) = ze^{at}$, we can
write
$$
X_t^*\big(\sigma^{(1)}(0)\bullet z + \sum_{j\ge2}
\frac{1}{j!}\,\sigma^{(j)}(0)\bullet z^j\big) =
\sigma^{(1)}(0)\bullet(ze^{at}) + \sum_{j\ge2}
\frac{1}{j!}\,\sigma^{(j)}(0)\bullet(ze^{at})^j \, , \, \forall\,t
\in \bc.$$

\noindent This gives a contradiction and proves the
lemma.\end{proof}

\noindent We also use

\begin{Lemma} \label{Lemma:nonnilpotent}  Let $X$ be a {\rm complete} holomorphic
vector field on a connected Stein surface $N$ with $H^2(N,\mathbb
Z)=0$ having a {\rm non-nilpotent} $($isolated$)$ singularity at
$P_0 \in N$. Then we have two possibilities for the first jet of
$X$ at $P_0$\,.
\begin{itemize}\itemsep=-1pt
\item[\rm(i)] $j_{P_0}^1\,X(x,y) = \la_1\,x\frac{\po}{\po x} +
\la_2y\,\frac{\po}{\po y}$\,, \,\, $\la_1\la_2 \ne 0$.
\item[\rm(ii)] $j_{P_0}^1\,X(x,y) = (\la_1x + y)\,\frac{\po}{\po
x} + \la_1\,y \frac{\po}{\po y}$\,, \, $\la_1 \ne 0$.
\end{itemize}
In particular $X$ is analytically linearizable around $P_0$ if
there are at least three separatrices at $P_0$\,.
\end{Lemma}

\begin{proof}  We divide the proof into several steps:

\vglue.1in

\noindent{\bf Step 1}:\, $j_{P_0}^1\,X$ has rank two.

\begin{proof}[Proof of Step 1]  Since $j_{P_0}^1\,X$ is non-nilpotent we
may assume that $j_{P_0}^1\,X(x,y) = \la_1y\,\frac{\po}{\po y}$
with $\la_1 \ne 0$. The (germ of the) vector field $X$ (at $P_0$)
is formally conjugate to a vector field $Y =
\la_1y\,\frac{\po}{\po y} + \frac{x^{k+1}}{1+\mu
x^k}\,\frac{\po}{\po x}$\,, with $k \ge 1$, $\mu \in \bc$. Since
$j_{P_0}^1\,X$ is not zero, according to \cite{[Rebelo]} $\fa(X)$
has two smooth and transverse invariant manifolds $\Ga_x$ and
$\Ga_y$ through $P_0$ tangent to the coordinate axes $\te_x$ and
$\te_y$ respectively. Again if we consider the leaf $L_0$
containing $\Ga_x-\{P_0\}$ then $L_0 \cong \bc^*$ and $\ov L_0 =
L_0 \cup \{P_0\} = L_0 \cup \Ga_x$ is parameterized by a map
(one-to-one) $\sigma\colon \bc \to \ov L_0$ with $\sigma(0) =
P_0$\,, and $X^* = X\big|_{\ov L_0}$ lifts to a vector field
$\widetilde X^*$ on $\bc$ with $\sigma_*(\widetilde X^*) = X^*$
and $\widetilde X^*(z) = az\,\frac{\po}{\po z}$ for some $a \in
\bc^*$. Therefore $\sigma_*\left(az\,\frac{\po}{\po z}\right)$ is
formally conjugate to the vector field $Y^* = \frac{x^{k+1}}{1+\mu
x^k}\,\frac{\po}{\po x}$ and therefore it is of the form
$$
\sigma_*\left(az\,\frac{\po}{\po z}\right) = \big(x^{k+1} +
h.o.t.\big)\,\frac{\po}{\po x} \quad (h.o.t. = \text{ higher order
terms})
$$
and this is not possible. Thus $j_{P_0}^1\,X$ has rank
two.\end{proof}

\noindent{\bf Step 2}:\, $j_{P_0}^1\,X$ is of type (i) or (ii) as
stated:

\begin{proof}[Proof of Step 2]  Since $j_{P_0}^1\,X$ has rank two we
conclude that

(1)\quad $j_{P_0}^1\,X = \la_1x\,\frac{\po}{\po x} +
\la_2y\,\frac{\po}{\po y}$\,, $\la_1\cdot\la_2 \ne 0$ or

(2)\quad $j_{P_0}^1\,X = (\la_1x+y)\,\frac{\po}{\po x} +
\la_2y\,\frac{\po}{\po y}$\,, \, $\la_1 \ne 0$.

\noindent By the classical theorem of Briot-Bouquet
 a germ of type (1) or (2) above admits a {\it
smooth\/} separatrix which is unique in case (2). In case (1) if
$\la_1/\la_2 \notin \re_-$ then the singularity is either resonant
with $\lambda_1/\lambda_2=n\in\mathbb N$  conjugate to the
Poincar\'e-Dulac normal form $Y_n$ or it is analytically
linearizable and either there is exactly one separatrix, there
exactly two separatrices or there are infinitely many and
$\la_1/\la_2 \in \bq_+$\,. Assume now that $\la_1/\la_2 \in \re_-$
then by \cite{[Martinet-Ramis]} there are exactly two
separatrices, which are smooth and transverse at $P_0$\,. This
proves Step 2.
\end{proof} The lemma follows from the three steps above.
\end{proof}

We can summarize the above discussion as follows:

\begin{Proposition} \label{Proposition:equivalences}  Let $X$ be a {\rm complete}
holomorphic vector field with finitely many singularities on a
connected Stein surface $N$ with $H^2(N,\bz) = 0$.
\begin{itemize}\itemsep=-1pt
\item[\rm(i)] Suppose $X$  exhibits a dicritical singularity $P_0
\in \sing(\fa(X))$ with non  degenerate linear part $($i.e.,
$DX(P_0)$ is nonsingular$)$ then there is a holomorphic
diffeomorphism $\Phi\colon N \to \bc^2$ which takes the flow to a
linear flow of $X_\la$ with $\la = \frac{k}{\ell}\in \mathbb Q_+$
on $\bc^2$.

\item[{\rm(ii)}] A non-nilpotent dicritical singularity of $X$ is
necessarily linearizable. This is the case of a non-nilpotent
singularity of $X$ having at least three separatrices. A simple
singularity of $X$ is necessarily non-nilpotent.

\item[\rm(iii)] Suppose $\fa(X)$ is transverse to the boundary
$\po D$ of some compact domain $D \subset N$ biholomorphic to the
closed unit disc $\ov D^2 (1) \subset \bc^2$ such that each leaf
of $\fa(X)$ crosses $\po D$. Then there exists a holomorphic
diffeomorphism $\Phi\colon N \to \bc^2$ conjugating $X$ to $X_\la$
$(\la \in \bc\backslash\re_-)$ or to $Y_n$ $(n\in\bn)$ on $\bc^2$.

\noindent As a consequence of {\rm(ii)} and {\rm(iii)} we obtain:
\item[\rm(iv)] Suppose  $\fa(X)$ is transverse to $\po D$ where $D
\subset N$ is biholomorphic to $\ov D^2 (1)$.  Then there is a
holomorphic diffeomorphism $\Phi\colon N \to \bc^2$ conjugating
$X$ to  $X_\la$\,, $\la = \frac{k}{\ell}$\,, $k,\ell \in \bn$,
provided that either of the following conditions is satisfied:
\begin{itemize}\itemsep=-1pt
\item[\rm(a)] $\fa(X)$ has a meromorphic first integral in a
neighborhood of $D$. \item[\rm(b)] $\fa(X)$ has infinitely many
closed leaves in $D\backslash\sing(\fa(X))$. \item[\rm(c)]
$\fa(X)$ has some dicritical singularity in $N$. \item[\rm(d)]
$\fa(X)$ has at least three closed leaves in
$D\backslash\sing(\fa(X))$. \item[\rm(e)] $\fa(X)$ has some closed
leaf in $D\backslash\sing(\fa(X))$ with finite holonomy group.
\end{itemize}
\end{itemize}
\end{Proposition}

\begin{proof}  Item (i) follows from the remark that any
leaf $L \subset \po B_{P_0}$ must accumulate one and only one
singularity $P(L)$ in $N\backslash\ov{B_{P_0}}$ and this
singularity is nondicritical (otherwise its basis of attraction
would be open). Thus, the singularity $P(L)$ has only finitely
many separatrices, what implies that the map $\po B_{P_0} \ni L
\mapsto P \in \sing(\fa(X))$ is a finite-to-one map and since
$\sharp\sing(\fa(X)) < \infty$ this implies that the number of
leaves on $\po B_{P_0}$ is finite. Now we apply the above
Lemma~\ref{Lemma:biholomorphictoplane} to conclude.

Item (ii) follows from Lemmas~\ref{Lemma:firstjet} and
\ref{Lemma:nonnilpotent}.  Item (iii) is stated in
Theorem~\ref{Theorem:main}. Item (iv) follows from (i), (ii) and
(iii) if we observe that conditions (a), (b), (c) and  (d) and (e)
are all equivalent.\end{proof}

\begin{proof}[Proof of Theorem~\ref{Theorem:equivalences}]
It is clear that (1) implies (2) because if $N$ is biholomorphic
to $\bc^2$ then we can equip $N$ with a complete holomorphic
vector field $X$ which is conjugate to $X_\lambda, \lambda \in
\mathbb C \setminus \mathbb R_-$ and also with a complete
holomorphic vector field $Y$ conjugate to $Y_n$. Let us prove that
(2) implies (1). Case (a) follows from the discussion in
\ref{subsubsection:Transversality}. Case (c) follows from
\ref{subsubsection:Dicriticalness} summarized in
Proposition~\ref{Proposition:equivalences} (i). Case (b) follows
from Proposition~\ref{Proposition:equivalences} (ii). Finally,
Case (d) in Theorem~\ref{Theorem:equivalences} follows from the
proof of Step 1 in the course of the proof of
Lemma~\ref{Lemma:boundary} or also as a straightforward
consequence of Proposition~\ref{Theorem:stability} below.
\end{proof}

\section{Proof of Theorem~\ref{Theorem:stability}:  global stability for holomorphic flows}
\label{section:stability}

As a consequence of the arguments in the preceding section  we
obtain the following global stability result for holomorphic flows
on connected Stein surfaces. We detail the arguments in  the proof of Theorem~\ref{Theorem:stability} announced in \cite{Scarduacomplexvariables}:

\begin{proof}[Proof of Theorem~\ref{Theorem:stability}]  Regarding
the hypothesis that $L_0$ is
closed we recall that, according to \cite{Suzuki1}, an orbit $L$
biholomorphic to $\bc^*$ is closed outside $\sing(\fa(X))$ and the
closure $\ov L_0 \subset N$ is an irreducible analytic curve in
$N$. Let $L_0$ be a closed orbit with finite holonomy group. Since
$N$ is Stein and $H^2(N,\bz) = 0$ there is a reduced equation
$\{f=0\}$ for $L_0$ in $N$, where $f\colon N \to \bc$ is
holomorphic.

Now we take a generator $\ga\colon S^1 \to L_0$ of the homology of
$L_0$ and a holomorphic one-form $\al$ in $L_0$ with the property
$\displaystyle\int_\ga \al = 1$. Since $L_0$ has finite holonomy
it follows that given a transverse disc $\sum \approx \bd$ to
$L_0$ with $\Sigma \cap L_0 = \{P_0\}$, $P_0 \notin
\sing(\fa(X))$, some power of the loop $\ga$ lifts as a closed
loop $\ga_z$ to the leaf $L_z$ of $\fa(X)$ through the point $z
\in \Sigma$ if $z$ is closed enough to $p_0$\,. Since $N$ is a
Stein manifold the one-form $\al$ is the restriction of a
holomorphic one-form $\tilde\al$ in $N$. For $z \in \Sigma$ close
enough to $p_0$ we have $\displaystyle\int_{\ga_z} \tilde\al \ne
0$.

Now, $\displaystyle\int_{\ga_z} \tilde\al =
\displaystyle\int_{\ga_z} \tilde\al\big|_{L_z}$ because $\ga_z
\subset L_z$ and since the restriction $\tilde\al\big|_{L_z}$ is
closed (because its holomorphic) we conclude that $L_z$ has
nontrivial homology for all $z \in \Sigma$ close enough to
$p_0$\,. Therefore, necessarily, $L_z \cong \bc^*$\, $\forall\,z
\in \Sigma$, close enough to $p_0$\,.  This implies that the
generic orbit of $X$ in $N$ is biholomorphic to $\bc^*$ and
$\fa(X)$ admits a meromorphic first integral on $N$.\end{proof}

\noindent According to what we have remarked in the beginning of
the above proof  we can equivalently state:

\begin{Proposition} \label{Proposition 5.9}  Let $X$ be a complete holomorphic
vector field in a connected Stein surface $N$ with $H^2(N,\bz) =
0$ and having an orbit $L_0$ biholomorphic to $\bc^*$,
accumulating no singularity of $\fa(X)$ and having finite holonomy
group. Then $\fa(X)$ admits a meromorphic first integral on $N$.
\end{Proposition}

Finally we also have

\begin{Corollary} \label{Corollary:stability}  Let $X$ be a complete holomorphic
vector field with finite singular set on a connected Stein surface
$N$ with $H^2(N,\bz) = 0$. Suppose that $X$ has infinitely many
orbits biholomorphic to $\bc^*$ with finite holonomy groups. Then
$\fa(X)$ admits a meromorphic first integral.
\end{Corollary}

\begin{proof}  We already know that if $\fa(X)$ has some
dicritical singularity then it admits a meromorphic first
integral. Assume therefore that each singularity of $\fa(X)$ is
nondicritical, i.e., exhibits only a finite number of
separatrices. Let $L$ be a leaf of $\fa(X)$ which contains some
separatrix of a singularity $p_0 \in \sing(\fa(X))$.  Then
necessarily $L$ is biholomorphic to $\bc^*$; otherwise $L \cong
\bc$ and $L \cup \{p_0\} \cong \ov\bc$ gives a compact curve in
$N$ what is not possible. Therefore $L$ is closed outside
$\sing(\fa(X))$ and $L \cup \{p_0\} \cong \bc$ so that $L$
contains no other separatrix of $\fa(X)$ and therefore we conclude
that a nonclosed leaf biholomorphic to $\bc^*$ of $\fa(X)$
accumulates exactly one singularity of $\fa(X)$. Since the numbers
of singularities and separatrices are finite it follows from the
hypothesis that $\fa(X)$ has some closed leaf biholomorphic to
$\bc^*$ with finite holonomy group and we may apply
Proposition~\ref{Theorem:stability} to conclude.\end{proof}

We close the paper with the following question:

\begin{Question}
Let $X$ be a complete holomorphic vector field  on $\bc^2$ with a
linearizable singularity of type $X_\lambda, \, \lambda \in \bc
\setminus \{0\}$ at the origin $0\in\bc^2$. Is $X$ linearizable on
$\bc^2$\,?
\end{Question}

\section{Proof of Theorem~\ref{Theorem:flowsstability}: flows with many periodic orbits}

Let $X$ be a {\rm complete} holomorphic vector field on a connected Stein surface $N$ such that $H^2(N,\bz) = 0$. Denote by $\xi(X)\subset N$ the subset of periodic orbits, i.e., orbits diffeomorphic to $\mathbb C^*$. Assume that $\xi(X)$ contains infinitely many orbits and that these orbits have some convergence orbit, say $L_0$. We have two possibilities:

\noindent{\bf Case 1}.  $L_0$ is a periodic orbit.
 In this case it is closed off $\sing(X)$. If it is not closed in $N$ then there is a single singular point $q \in \sing(X)$ which is accumulated by $L_0$. If $q$ is dicritical then, as we have already see, there is a meromorphic first integral for $\fa(X)$. Assume then that $q\in \sing(X)$ is not dicritical. Since $q$ is a convergence point for  $\xi(X)$ and $q$ is non-dicritical, the germ of foliation $(\fa(X),q)$ induced by $\fa(X)$ at $q$, admits  infinitely many closed  leaves. By \cite{Camacho-Scardua} this implies that either $\fa(X)$ admits a holomorphic first integral, or it is a formal pull-back of a (nRnL) singularity. In the holomorphic first integral case it is easy to check that the flow is periodic in a neighborhood of the singularity and, therefore, by the Identity Principle, it is a periodic flow, corresponding  to a $\mathbb C^*$-action.
 Using now the main result in \cite{Camacho-Scardua} we conclude that the flow is a holomorphic pull-back of a linear periodic flow, admitting a holomorphic first integral of type $f^n g ^m$, for some $n,m \in \mathbb N$ and some $f, g$ holomorphic in $N$. Assume now that we are in the (nRnL) case. Then the  holonomy group of the separatrix germ of the singularity $q$ contained in the closure of the  orbit $L_0$, is a (nRnL) map. Since $L_0$ is topologically a cylinder $S^1 \times \mathbb R$, the fundamental group of $L_0$ is generated by some rational power of the loop $\gamma$ that originates the local separatrix holonomy. Hence, the holonomy group of $L_0$ is generated by this (nRnL) map, proving that $L_0$ is a (nRnL) orbit.
 Finally, if $L_0$ is closed in $N$ then, as above,  there infinitely many closed orbits for its holonomy group (no matter what the orbits in $\xi(X)$ converging to $L_0$ accumulate or not some singularity of $X$). Thus, according to Lemma~\ref{Lemma:perezmarcocyclic} this implies that the holonomy group of $L_0$ is either a finite group or it is generated by a periodic map and a (nRnL) map. In the finite case Theorem~\ref{Theorem:stability} implies that there is a meromorphic first integral. In the remaining case we have a (nRnL) orbit.

\vglue.1in

\noindent{\bf Case 2}. $L_0$ is a simply-connected orbit.
In this case if $L_0$ is not closed in $N$ then there are two possibilities.
Either $L_0$ accumulates at some singularity $q \in \sing(X)$, or $L_0$ accumulates at some orbit $L_0 ^\prime$ in $N$. If $L_0$ only accumulates at singularities, then by Remmert-Stein theorem, the closure $\ov{L_0}\subset $ is an analytic subset of $N$, of dimension one. Nevertheless, this set is topologically isomorphic to a Riemman sphere, therefore it is compact, which is not possible in a Stein manifold. Thus $L_0$ cannot accumulate only at singularities and must accumulate at some orbit $L_0 ^\prime$.
Since $L_0$ is a convergence orbit for $\xi(X)$, we conclude that the orbit $L_0 ^\prime$ is a convergence orbit for $\xi(X)$.
If the orbit $L_0^\prime$ is periodic then
we have again a periodic orbit in the convergence set of $\xi(X)$ and we can report to Case 1 above.
Assume now that $L_0^\prime$ is a simply-connected orbit. If $L_0^\prime$ accumulates at some periodic orbit $L_0^{\prime \prime}$ then we are in the same situation above, replacing $L_0^\prime$ by $L_0^{\prime \prime}$. The remaining possibility is that   $L_0^\prime$  accumulates only at simply-connected orbits. Finally, if $L_0$ is closed in $N$, then we have a simply-connected closed orbit in the convergence set of $\xi(X)$.

\section{Proof of Theorem~\ref{Theorem:affinesurface}: flows  on affine surfaces}
In this section we study holomorphic flows on affine surfaces. By an {\it affine surface} we mean an algebraic surface embedded in some affine space $\mathbb C^m$. Such a complex manifold is Stein. A classical result due to Kodaira states that any Stein manifold can be embedded into some affine space. We consider the case
where the manifold is algebraic. Also, we consider the case of algebraic vector fields, i.e., vector fields given by polynomial coordinates on the affine surface. Let $X$ be an algebraic vector field on an affine surface $N\subset \mathbb C^m$. Denote by $\ov{N}\subset \mathbb CP^m$ the projectivization of $N$, i.e., the corresponding projective variety $\ov{N}\subset \mathbb C P^m$ such that $\ov{N}\cap \mathbb C^m= N$. For sake of simplicity, we also assume that $\ov{N}$ is non-singular, so that this is a complex manifold with the natural structure. We shall refer to this by saying that $N$ is {\it regular}. Since $X$ is algebraic, the foliation $\fa$  induced by $X$ on $N$ extends to a holomorphic foliation with singularites $\ov{\fa}$ on $\ov{N}$.
Such a foliation is given by a rational one-form $\ov{\omega}$ on $\ov{N}$.
We shall refer to the pair $(\ov{N}, \ov{\fa})$ as the {\it projectivization} of the pair $(N,\fa)$.
The divisor $D:= \ov{N}\setminus N$ is a finite union of projective curves, each of which may be invariant or not by $\ov{\fa}$. Recall that a singularity
is a {\it generalized curve\/} if its reduction of singularities (cf. \cite{C-LN-S1}, \cite{seidenberg}) by the blowing-up process, produces only non degenerate (i.e., no saddle-node) singularities \cite{C-LN-S1}. Thus after the reduction of singularities, all singularities are of the local irreducible form
 $xdy - \la ydx +...=0$ and $\la\notin{\mathbb Q}_+$. We are therefore excluding one of the possible  final products of the reduction of singularites: the case where the singularity is a {\it saddle-node}, which is a singularity of the form $y^{k+1}dx -[x(1 + \la y^k)+\hot]dy=0$, $k\ge 1$.
\vglue.1in
Our main result for this situation is Theorem~\ref{Theorem:affinesurface} which we prove in what follows. 
\begin{proof}[Proof of Theorem~\ref{Theorem:affinesurface}]
Let us consider a  singularity $q \in D\cap \sing(\ov{\fa})$. Given a local separatrix $\Gamma_q$ of $\ov{\fa}$ through $q$, which is not contained in some local branch of $D$, we have two possibilities for the orbit $L(\Gamma_q)$ of $X$ that contains $\Gamma_q^*:=\Gamma_q \setminus \{q\}$:
(i) $L(\Gamma_q)$ is  a simply-connected  orbit. In this case the union $\Gamma_q \cup L(\Gamma_q)$ is an invariant analytical rational algebraic curve $C(\Gamma_q)\subset \ov{N}$.
(ii) $L(\Gamma_q)$ is a periodic orbit.

Since the number of local branches of $D$ through any point $q\in D$ is finite we conclude that if $q$ is dicritical then $\ov{\fa}$ has infinitely many invariant rational curves in $\ov{N}$ or it there is  an open subset in $N$ (the saturation of a sector with vertex at $q$) where the orbits of $X$ are all periodic.
In the first case by Darboux theorem (\cite{Jo}) there is a rational first integral for $\ov\fa$. In the second case, by Suzuki the generic orbit of $X$ is periodic and there is a meromorphic first integral for $\fa$ in $N$ \footnote{In this case, by the main result of \cite{Mol}, under some mild conditions on the singularities in $D$, the first integral extends
as a rational first integral for $\ov{\fa}$ in $\ov{N}$. See the remark in the end of this section.}.

From now on we assume therefore that all the singularities of $\ov{\fa}$ are
non-dicritical. Also we may assume that  all the irreducible components of $D$ are invariant by $\ov{\fa}$.

As usual we denote  by $\xi(X)\subset N$ the subset of periodic orbits of $X$ in $N$. By Theorem~\ref{Theorem:flowsstability} we have two cases (notice that $\ov N$ is compact so that the set of convergence points of $\xi(X)$ is not empty in $\ov{N}$):
\vglue.1in
\noindent{\bf Case 1}. $X$ has some meromorphic first integral in $N$.
\vglue.1in
\noindent{\bf Case 2}. All the singularities of $X$ in $N$ are non-dicritical, $\xi(X)$ is a diverging set in $N$ and no closed periodic orbit in $N$ has
finite holonomy.

\vglue.1in
Let us first consider Case 2.
In this second case, since $\ov{N}$ is compact and $\xi(X)$ contains infinitely many orbits, there is some irreducible component $D_0\subset D$ which is contained in the set of convergence points of $\xi(X)$. Because all the singularities in $D$ are non-dicritical generalized curves (i.e., non-degenerate after the reduction of singularities with invariant exceptional divisor), and since $D$ is invariant, it follows that all components in $D$ are contained in the set of convergence points of $\xi(X)$. For sake of simplicity, we shall assume that the singularities in $D$ are already irreducible, so that we may assume that all the singularities in $D$ are non-degenerate. This does not imply any loss of generality. Indeed, this would be the final picture after the reduction of singularities by a finite number of quadratic blow-ups in the singularities of $\ov \fa$ in $D$ and we are considering the case where these singularities are non-dicritical. Therefore, the exceptional divisor together with the strict transform of $D$ would be connected and invariant and also a very ample divisor in the corresponding projective manifold obtained from $\ov N$.
Adopting this hypothesis  we proceed.
Let us fix a component $D_0\subset D$.
We introduce the {\it virtual holonomy} group of this component $\Hol^{\virt}(\ov{\fa}, D_0)$ as follows:
Fix a point $p\in D_0$ which is not an intersection point of several components of $D$, nor a singular point for $\ov{\fa}$. Given a transverse disc $\Sigma_p$ to $\fa$ centered at $p= \Sigma_p \cap D_0 = \Sigma \cap D$.
The {\it virtual holonomy group} of the leaf with respect to the transverse section $\Sigma_p$ and base point $p$ is defined as (\cite{C-LN-S2})
\[
\Hol^{\virt} (\fa, D_0, \Sigma_p,p)=\{ f \in\Diff
(\Sigma_p, p) \big|\tilde L_z = \tilde L_{f(z)}.
\forall z\in (\Sigma_p, p) \}
\]
The virtual holonomy group contains the holonomy group $\Hol(\fa,  L_p, \Sigma_p, p)$ of the leaf $p \in L_p = D_0 \setminus \sing(\fa)$, and consists of the map germs that preserve each of the leaves of the foliation.

\begin{Lemma}
\label{Lemma:closedleaf}
Let $\fa$ be holomorphic foliation defined in a neighborhood  $U$ of the origin  $0 \in \mathbb C^2$, with an isolated singularity at the origin. A closed leaf of $\fa$ in $U$ is analytic. A leaf that accumulates only at the singular point is contained in an invariant analytic curve. Let $\Gamma\subset U$ be a separatrix of the foliation through the origin, let $p \in \Gamma \setminus \{0\}$ and $\Sigma_p$ a small disc transverse to the foliation and centered at $p$. Then any closed leaf of $\fa$ intersecting $\Sigma_p$ induces a closed orbit for the holonomy group $\Hol(\fa, L_p, \Sigma_p, p)$ of the leaf $L_p$ and the same holds for the virtual holonomy group $\Hol^{\virt}(\fa, L_p, \Sigma_p, p)$.
\end{Lemma}
\begin{proof}
Indeed, by Remmert-Stein extension theorem (\cite{gunning2}), a leaf which is closed in $U$ is analytic. Also by Remmert-Stein extension theorem, a leaf $L$ such that $\overline{L}\setminus L = \{0\}$ has analytic closure, because $\dim L=1 > 0=\dim(\overline{L}\setminus L)$. The last part follows from the fact that the intersection of two transverse analytic sets of dimension one in $\mathbb C^2$, is a closed set of points.
\end{proof}

Using then the above argumentation we conclude that:
\begin{Lemma}
Each virtual holonomy group of the irreducible components of $D$ is a group with infinitely many closed pseudo-orbits.
\end{Lemma}

Therefore, each such group is either finite or it contains some (nRnL) map and
it is abelian formally but not analytically linearizable. Let $D_0$ and $D_0 ^\prime$ be two distinct components of $D$, intersecting at a {\it corner} $q=D_0 \cap D_0^\prime$. Since both components are invariant, this corner is a singularity of $\ov\fa$. Each of the components corresponds to a local separatrix of $\ov{\fa}$ through $q$. By the above lemma, each such local holonomy is a finite periodic map or a formally linearizable non-resonant map. We claim:
\begin{Claim}
Either both virtual holonomy groups of $D_0$ and $D_0^\prime$ are finite, or
both are abelian formally linearizable, both containing a (nRnL) map.

\end{Claim}
\begin{proof}
Assume that the virtual holonomy group of $D_0$ is finite. Then, the singularity $q=D_0 \cap D_0 ^\prime$ is a linearizable singularity of the local form $nxdy + mydx=0$ with $n, m \in \mathbb N$, in suitable coordinates $(x,y)$ centered at $q$, such that the components correspond to the coordinate axes. Assume by contradiction that the virtual holonomy group of $D_0^\prime$ is not finite. Let us state that $D_0$ is given by $(y=0)$ and $D_0^\prime$ by $(x=0)$. Choose a disc $\Sigma: \{x=a\}$, for $|a|>0$ small enough.
There is some element $f(y)$ in this virtual holonomy group which is a (nRnL) map. Because $f$ commutes with the local holonomy corresponding to the separatrix $(x=0)$ of $q$ contained in $D_0^\prime$ we conclude that $f(x)=\lambda x (1 + \psi(x^m)$, for some holomorphic one-variable function $\psi(z)$, vanishing at the origin $0 \in\mathbb C$. Thus we can induce a virtual holonomy map $f^{\mathcal D}\in \Hol^{\virt} (\ov{\fa}, D_0, \Sigma, a)$ by setting $f^{\mathcal D} (y) = \lambda ^{\frac{m}{n}}y (1 + \psi(y^n))^{\frac{m}{n}}$.
\begin{Remark}[Dulac map associated to a corner (cf. \cite{Camacho-Scardualiouville,camacho-scarduaasterisque,scarduaJDCS,scarduabounded})]
{\rm
  The local leaves of the foliation are given by $x^m y^n=cte$.
The Dulac correspondence is the correspondence obtained by following these  leaves $$
\mathcal D_{q}  \colon \Sigma_0 \to \Sigma_0^ \prime, \, \,\mathcal D_{q}(x)= \{x^{{m}/{n}}\}.
$$
from a local transverse section $\Sigma_0$ to $D_0$ to another transverse section $\Sigma_0 ^\prime$ to $D_0 ^\prime$.
Thus, we have associated to a map $f$ in the virtual holonomy of $D_0^\prime$, a well-defined  map  $f^{\mathcal D_{q}}$ in the virtual holonomy of $D_0$, such that it
satisfies the "equation"
\[
f^{\mathcal D_q}\circ\mathcal D_{q}=\mathcal D_{q}\circ f.
\]
}
\end{Remark}

Since the original map $f$ is a (nRnL) map, the same holds for the map $f^{\mathcal D_q}$. Thus the virtual holonomy group of $D_0$ also contains a (nRnL) map and therefore it cannot be finite, contradiction. This proves the claim.
\end{proof}

An iteration of this argument actually shows that:

\begin{Claim}
Either all virtual holonomy groups of components of $D$  are finite, or
all such virtual holonomy groups are  abelian formally linearizable,  containing  (nRnL) maps.
\end{Claim}
In case all virtual holonomy groups are finite, using the same argumentation in
\cite{camacho-scarduadarbouxlocal} we conclude that there is a holomorphic first integral for the foliation in a neighborhood $W$ of $D$ in $\ov{N}$. In the remaining case, also as in \cite{camacho-scarduadarbouxlocal} we can conclude that there is a transversely formal closed meromorphic one-form $\hat \Omega$, defined on the divisor $D$, with simple poles, and which satisfies $\hat \Omega \wedge \ov{\omega}=0$.

Now, we use the fact that
since $\sing \ov{\fa} \cap D$ is nondicritical it follows that $D \subset {\ov N}$ is a very ample algebraic divisor  in the projective manifold $\ov N$ (see \cite{Hi-Mt} for the definition of very ample curve). This is a
consequence of Levi's Theorem \cite{Si}, \cite{C-LN-S2}. Therefore, according to
Hironaka-Matsumura Theorem \cite{Hi-Mt} $\hat\Omega$ is in fact a closed rational one-form on
$\ov N$ with simple poles.

As for the Case 1., where there is a meromorphic first integral for $\fa$ in $N$. We conclude, from Lemma~\ref{Lemma:closedleaf}, that all virtual holonomy groups of the components of $D$ are finite and as above, there is a holomorphic first integral for $\ov\fa$ defined in a neighborhood $W$ of $D$.
Now, since $\ov{N}\setminus D=N\subset \mathbb C^m$ is a Stein manifold, this implies, as in \cite{C-LN-S2}, that any meromorphic function or form, defined in a neighborhood of $D$ in $\ov{N}$ extends to $\ov{N}$ as a rational function or form. Therefore, we conclude that, in Case 1, as well as in the remaining possibility for Case 2, there is a rational first integral for $\ov{\fa}$ defined in $\ov{N}$.
This ends the proof of Theorem~\ref{Theorem:affinesurface}.
\end{proof}

\begin{Remark}
\label{Remark:complements}
{\rm Regarding the statement of Theorem~\ref{Theorem:affinesurface} we can add:
\begin{enumerate}
\item The case where the number of periodic orbits of $X$ in $N$ is finite is also addressable. Indeed, in this case if there is some dicritical singularity then $\ov \fa$ exhibits infinitely many invariant rational curves (i.e., compact curves diffeomorphic to the Riemann sphere $\ov {\mathbb C} = \mathbb C \cup \{\infty\}$). By Darboux theorem the foliation admits a rational first integral. Thus we may assume that the singularities are non-dicritical. In this case the  holonomy groups of the components of the divisor $D$ are groups of germs of holomorphic diffeomorphisms without fixed points: indeed, any fixed point outside the origin, of a certain holonomy map,  corresponds to a closed lifting of a closed path in the base leaf, generating this holonomy map. The leaf containing the fixed point  cannot be simply-connected, contradiction.
Thus, by Nakai density theorem these holonomy groups are solvable. Using then the techniques developed in \cite{Camacho-Scardualiouville} and \cite{scarduaJDCS} we can show that the foliation is given by a closed rational one-form with simple poles also in this case.

\item As for the case the foliation $\fa$ exhibits a meromorphic first integral on $N$, we can study its extension to a rational first integral on $\ov N$ as follows: from the work of Mol in \cite{Mol}, under some mild conditions on the singularities on $D$, there is some singularity in $D$ exhibiting a separatrix not contained in $D$. In this case, the extension of the first integral to $D$ is made starting from the separatrix to the irreducible component of $D$ that contains the referred singularity and then to all the singularities contained in this component. Then we extend the first integral to the corners and to the adjacent components of $D$, and so on. Thus, the only case where the first integral cannot be extended to $D$, is when $D$ contains all the separatrices of singularities of $\ov\fa$ in $D$. Nevertheless, in this case, according to the work of Licanic (\cite{licanic}), under some conditions on the Picard group $Pic(\ov N)$, we can assure that the foliation is a logarithmic foliation, i.e., given by a simple poles closed rational one-form.

\end{enumerate}
}
\end{Remark}
\bibliographystyle{amsalpha}

\end{document}